\documentclass{amsart}
\usepackage{amsthm,amsfonts,amsmath,amscd,amssymb,latexsym,epsfig}
\newcommand{\gl}{{\mathfrak g \mathfrak l}}

\renewcommand{\u}{{\mathfrak u}}

\newcommand{\g}{{\mathfrak g}}         
\newcommand{\h}{{\mathfrak h}}         

\newcommand{\cx}{{\mathbb C}}

\newcommand{\Ad}{\operatorname{Ad}}

\newcommand{\Hom}{\operatorname{Hom}}

\newcommand{\End}{\operatorname{End}}
\newcommand{\Aut}{\operatorname{Aut}}

\newcommand{\rank}{\operatorname{rank}}

\newcommand{\Jac}{\operatorname{Jac}}
\newcommand{\Mor}{\operatorname{Mor}}

\newcommand{\Mat}{\operatorname{Mat}}

\newcommand{\ol}{\overline}

\numberwithin{equation}{section}

\newtheorem{theorem}{Theorem}[section]

\newcounter{mtheorem}

\newtheorem{mtheorem}[mtheorem]{Theorem}

\newtheorem{lemma}[theorem]{Lemma}

\newtheorem{proposition}[theorem]{Proposition}

\theoremstyle{remark}

\newtheorem{remark}[theorem]{Remark}

\newtheorem{definition}[theorem]{Definition}

\newtheorem{example}[theorem]{Example}

\newcommand{\oP}{{\mathbb{P}}}

\newcommand{\sH}{{\mathcal{H}}}

\newcommand{\sK}{{\mathcal{K}}}
\newcommand{\sM}{{\mathcal{M}}}   

\newcommand{\sO}{{\mathcal{O}}}

\newcommand{\fH}{{\mathfrak{h}}}

\newcommand{\fZ}{{\mathfrak{z}}}

\begin{document}

\title{Invariant hypercomplex structures and algebraic curves}
\author{Roger Bielawski}
\address{Institut f\"ur Differentialgeometrie,
Leibniz Universit\"at Hannover,
Welfengarten 1, 30167 Hannover, Germany}



\begin{abstract}  We show that $U(k)$-invariant hypercomplex structures on (open subsets) of regular semisimple adjoint orbits in $\gl(k,\cx)$ correspond  to algebraic curves $C$ of genus $(k-1)^2$, equipped with a flat projection $\pi:C\to\oP^1$ of degree $k$, and  an antiholomorphic involution  $\sigma:C\to C$ covering the antipodal map on $\oP^1$.                            
\end{abstract}

\maketitle

\thispagestyle{empty}

Invariant hyperk\"ahler metrics on adjoint orbits of a reductive complex Lie group $G^\cx$ have been constructed by Kronheimer \cite{Kron1,Kron2} in the case of nilpotent and semisimple orbits, and by Biquard \cite{Biq} and Kovalev \cite{Kov} for arbitrary orbits. In the case of regular orbits, the metrics constructed by these authors are parametrised by triples $(\tau_1,\tau_2,\tau_3)$ of elements of a Cartan subalgebra. 
For a generic $\zeta\in \oP^1$, this hyperk\"ahler manifold $M(\tau_1,\tau_2,\tau_3)$ is isomorphic, as a complex symplectic manifold, to the regular adjoint $G^\cx$-orbit of an $l_\zeta\in \g^\cx $, the semisimple part of which equals  $(\tau_2+i\tau_3)+2i\tau_1\zeta +(\tau_2-i\tau_3)\zeta^2$. D'Amorim Santa-Cruz \cite{S-C} constructed a much larger class of $G$-invariant pseudo-hyperk\"ahler metrics on (open subsets of) regular adjoint orbits, parametrised by an arbitrary real spectral curve $S$, i.e.\ a real section of $\bigl(\g^\cx\otimes\sO_{\oP^1}(2)\bigr)/G^\cx$. This manifold $M(S)$ has several connected components; in particular there is always a component where the metric is positive-definite, and a component  where the metric is negative-definite. A rather surprising (at least to the authors) result in \cite{BF} is that already for $G^\cx=SL(3,\cx)$ there exist components on which the metric is indefinite. 
\par
Twistor theory implies almost immediately that any $G$-invariant and locally $G^\cx$-homogeneous 
pseudo-hyperk\"ahler manifold can be obtained from the  D'Amorim Santa-Cruz construction. The main purpose of this article is to show that this is no longer the case for  hypercomplex manifolds. Let $Q$ denote the $2$-sphere of complex structures defining the hypercomplex structure of a hypercomplex manifold. We prove:
\begin{mtheorem} Let $(M,Q)$ be a connected hypercomplex manifold with a free triholomorphic action of $PU(k)$ and such that:
\begin{itemize}
\item[(i)] for any complex structure in $Q$, the local action of $PGL(k,\cx)$ is transitive with the infinitesimal stabiliser isomorphic to the centraliser of a regular element;
\item[(ii)] for a generic complex structure in $Q$, the stabiliser in (i) is reductive (i.e. a Cartan subalgebra).
\end{itemize}
Then there exists a connected reduced lci algebraic curve $C$ of (arithmetic) genus $g=(k-1)^2$, equipped with a  flat projection $\pi:C\to\oP^1$ of degree $k$ and  an antiholomorphic involution $\sigma$ covering the antipodal map on $\oP^1$, such that $M/PU(k)$ is canonically isomorphic to an open subset of 
$\bigl(\Jac^{g-1}(C)\backslash(\Theta\cup\Delta)\bigr)^\sigma$, where $\Theta$ is the theta divisor and $\Delta$ is the divisor of invertible sheaves $L$ of degree $g-1$ such that the shifted Petri map $$H^0(C,L(1))\otimes H^0(C,L^\ast(1)\otimes K_C)\to H^0(C,K_C(2)),\enskip \text{where $L(i)=L\otimes\pi^\ast\sO_{\oP^1}(i)$},$$ is not an isomorphism. The action of $\sigma$ on $\Jac^{g-1}(C)$ is given by $L\mapsto \bigl(\ol{\sigma^\ast L}\bigr)^\ast\otimes K_C$.
\par
Conversely, given $\pi:C\to\oP^1$ as above, there exists a canonical $PU(k)$-invariant hypercomplex structure, satisfying (i)-(ii), 
on  a principal $PU(k)$-bundle over an open subset of 
$\bigl(\Jac^{g-1}(C)\backslash(\Theta\cup\Delta)\bigr)^\sigma$.
\end{mtheorem}
The (real) dimension of the moduli space of such $(C,\pi,\sigma)$ is $2k^2-2k$.
The curves in the D'Amorim Santa-Cruz construction are embeddable in $T\oP^1$, and therefore satisfy $h^0(C,\pi^\ast\sO_{\oP^1}(2))\geq 4$. However, a generic curve of genus $(k-1)^2$  with a flat projection of degree $k$ to $\oP^1$ has $h^0(C,\sO(2))=3$ if $k\geq3$. Therefore we can conclude that for $k>2$ a generic $U(k)$-invariant locally $GL(k,\cx)$-homogeneous hypercomplex manifold is not pseudo-hyperk\"ahler. In fact, these hypercomplex manifolds cannot even admit a hypercomplex moment map: as soon as they do, they must be  pseudo-hyperk\"ahler. 
\par
We also remark that a canonical hypercomplex structure exists on a principal bundle over any connected component $\bigl(\Jac^{g-1}(C)\backslash(\Theta\cup\Delta)\bigr)^\sigma$. The structure group (and hence the symmetry group of the hypercomplex structure) will, however, usually vary between different components of $\bigl(\Jac^{g-1}(C)\backslash\Theta\bigr)^\sigma$ (among different real forms of $PGL(k,\cx)$).

\section{Background material}

\subsection{Hypercomplex structures and generalisations}

A hypercomplex structure on a smooth manifold $M$ is given by a subbundle of $\End(TM)$ which is isomorphic to quaternions and the almost complex structures corresponding to unit quaternions are all integrable. Equivalently, a hypercomplex structure is a decomposition $T^\cx M\simeq E\otimes \cx^2$ of the complexified tangent bundle into the trivial rank $2$ bundle and a quaternionic bundle $E$. Moreover the tensor product of the quaternionic structure on $E$ and the standard quaternionic structure on $\cx^2$ ($(z,w)\mapsto (-\bar w, \bar z)$ is the complex conjugation on $T^\cx M$. A (pseudo)-hyperk\"ahler manifold is a hypercomplex manifold equipped with a compatible (pseudo)-Riemannian metric. Both hypercomplex and
hyperk\"ahler manifold arise via twistor theory as spaces of $\sigma$-invariant (known as real) sections of a holomorphic submersion 
$\pi:Z\to \oP^1$ with normal bundle $\bigoplus \sO(1)$, where $\sigma:Z\to Z$ is a holomorphic involution covering the antipodal map on $\oP^1$.
\par
We are going to consider more general structures. For one, it is useful to consider complex analogues of these notions. Thus a $\cx$-hypercomplex manifold is a complex manifold $M$ such that its holomorphic tangent bundle decomposes as $E\otimes \cx^2$ and, for each nonzero $v\in \cx^2$, the distribution $E\otimes v$ is integrable. This is the geometry on the space of all sections of the twistor fibration with normal bundle $\bigoplus \sO(1)$. Secondly, we want to relax the hypercomplex condition itself, so that we can describe the geometry on a larger part of the Kodaira moduli space of sections of $\pi:Z\to \oP^1$, where the normal bundle jumps.  As argued in \cite{BP1,BP3}, the relevant geometry is that of $2$-Kronecker structures.
\par
A {\em $2$-Kronecker} structure (of rank $2n$) on a complex manifold $M^{4n}$ consists of a holomorphic vector bundle $E$ of rank $2n$ on $M$ and a bundle map $\alpha:E\otimes \cx^2\to TM$ which is injective on $\alpha\rvert_{E_m\otimes v}$ for every $m\in M$ and $v\in\cx^2$.
\par
A quaternionic $2$-Kronecker structure on a real manifold $M^{4n}$ consists of a quaternionic bundle $E$ of (complex) rank $2n$ and a bundle map  $\alpha:E\otimes \cx^2\to T^\cx M$ which is injective on $\alpha\rvert_{E_m\otimes v}$ for every $m\in M$ and $v\in\cx^2$, and intertwines the tensor product of quaternionic structures of $E$ and $\cx^2$ with complex conjugation on $T^\cx M$.
\par
A $2$-Kronecker structure (quaternionic or not) is {\em integrable}, if $T^v M=\alpha(E\otimes v)$ is an involutive subbundle for every $v\in \cx^2\backslash \{0\}$.
\par
Thus hypercomplex (resp.\ $\cx$-hypercomplex) structures are integrable quaternionic $2$-Kronecker structure (resp.\  integrable $2$-Kronecker structures) such that $\alpha$ is an isomorphism.

\subsection{Pseudo-hyperk\"ahler metrics of D'Amorim Santa-Cruz\label{SC}}
Let $G$ be a compact Lie group and $G^\cx$ its complexification. 
 We denote by $\pi$ the natural projection 
\begin{equation}\pi: \g^\cx\otimes \sO(2)\to \bigl(\g^\cx\otimes \sO(2)\bigr)/G^\cx\simeq \bigl(\fH^\cx\otimes\sO(2)\bigr)/W\simeq\bigoplus_{i=1}^r\sO(2d_i),\label{pi}\end{equation}
where $r=\rank G$ and the $d_i$ are the degrees of $\Ad G^\cx$-invariant polynomials forming a basis of $\cx[\g]^{G^\cx}$. We denote by $\bar\mu$ the composition of $\mu$ with the map induced by $\pi$ on global sections. A regular adjoint $G^\cx$ orbit corresponds to a point in $\h^\cx/W$. Therefore, if a section $A(\zeta)$ of $\g^\cx \otimes \sO(2)$ is regular for every $\zeta$, then its $G^\cx$-orbit is identified with a section of $\bigl(\fH\otimes\sO(2)\bigr)/W\simeq\bigoplus_{i=1}^r\sO(2d_i)$.  We shall call any such section $S$ a {\em spectral curve}. 
\begin{definition} A section $A(\zeta)=A_0+A_1\zeta+A_2\zeta^2$ of $\g^\cx \otimes \sO(2)$ is called {\em regular} if
$A(\zeta)$ is regular for every $\zeta$. It is called {\em strongly regular} if it is regular and 
the centralisers of $A(\zeta)$ span,  as $\zeta$ varies in $\oP^1$, the whole $\g^\cx$.
\label{strreg}
\end{definition}
\begin{remark} For $G=U(k)$ strong regularity means that the coefficients of the powers $A(\zeta)^i$, $i=1,\dots,k-1$ are linearly independent in the space of complex $(k\times k)$-matrices.
\end{remark}
Let $X_S$ be the submanifold of $\pi^{-1}(S)$ consisting of regular elements of $\g^\cx\otimes\sO(2)$.
Strongly regular sections of $X_S$ which are also real in the sense that $A(\zeta)=(T_2+iT_3)+2iT_1\zeta +(T_2-iT_3)\zeta^2$, where $T_i\in \g$, are called {\em twistor lines} by D'Amorim Santa-Cruz. They form a manifold $M(S)$, shown by D'Amorim Santa-Cruz \cite{S-C} to be pseudo-hyperk\"ahler. Indeed, he shows that $M(S)$ has  a real bilinear form $g_0$, compatible with hypercomplex structure, such that corresponding holomorphic $2$-forms on fibres of $X_S$ (e.g. $g_0(I_2\,\cdot,\cdot)+ig_0
(I_3\,\cdot,\cdot)$ for $I_1$) coincide with the Kostant-Kirillov-Souriau symplectic forms on adjoint orbits. This means that the fundamental forms $g_0(I_i\,\cdot,\cdot)$, $i=1,2,3$, are nondegenerate and, consequently $g_0$ is a pseudo-hyperk\"ahler metric. 
\begin{remark} If we consider all strongly regular sections, not just the real ones, we obtain a complex manifold equipped with a $\cx$-hyperk\"ahler structure. The $\cx$-hypercomplex structure of this manifold extends to a $2$-Kronecker structure on the manifold of all regular sections. Indeed, it follows from the arguments of D'Amorim Santa-Cruz that $H^1(A,N(-1))=0$ for any regular section, where $N$ is the normal bundle of $A$ in $X_S$. Therefore the degrees of rank one direct summands of $N$ can be only $0,1,2$.  
\end{remark}

\section{Hypercomplex structures from algebraic curves}

Let $C$ be a connected and reduced lci algebraic curve of genus $g$ and $L$ a globally generated and non-special (i.e.\  $h^1(L)=0$) invertible sheaf of degree $d=g+k-1$, $k\geq 2$. Then $h^0(L)=k$, and for any basis $\underline{s}=(s_1,\dots,s_k)$ of $H^0(C,L)$, we obtain a morphism $\phi_L=\phi_{L,\underline{s}}:S\to \oP^{k-1}$ defined by
$$ x\longmapsto [f_1(x),\dots,f_k(x)]\in \oP^{k-1},$$
where $s_i(x)=f_i(x)e$,  $i=1,\dots,k$, for some local holomorphic section $e$ of $L$. Clearly changing the basis $\underline{s}$ by an overall multiplicative scalar does not change $\phi_{L,\underline{s}}$.
\par
Conversely, if $\phi:C\to  \oP^{k-1}$ is a holomorphic map such that the Hilbert polynomial $p(m)=\chi\bigl(\phi^\ast \sO_{\oP^{k-1}}(m)\bigr)$ is equal to $dm-g+1$, and $\phi(C)$ is not contained in a hyperplane, then $L=\phi^\ast \sO_{\oP^{k-1}}(1) $ has degree $d=g+k-1$ and $\phi=\phi_{L,\underline{s}}$
for the basis $\phi^\ast z_i$, where $(z_i)$ is the standard basis of  $H^0\bigl( \oP^{k-1},\sO_{\oP^{k-1}}(1) \bigr)$.
\par
We denote the moduli space of invertible sheaves of degree $l$ by $\Jac^l(C)$, and set:
$$ \Jac_0^{g+k-1}(C)=\{L\in \Jac^{g+k-1}(C);\;\text{$L$ is globally generated},\enskip h^1(L)=0\}.$$
Moreover, $F_0^{g+k-1}(C)$ will denote the principal $PGL(k,\cx)$-bundle over $\Jac_0^{g+k-1}(C)$, the fibres of which consist of frames of $H^0(C,L)$ modulo scalars. We have a bijection
\begin{equation} F_0^{g+k-1}(C)\stackrel{1-1}{\longrightarrow}\Mor_0^p(C,\oP^{k-1}),\label{isom}\end{equation}
where $\Mor^p_0(C,\oP^{k-1})$ is the Hilbert scheme of nondegenerate (i.e.\  such that the image of $C$ is not contained in a hyperplane) morphisms with Hilbert polynomial $p(m)=dm-g+1$.
\begin{proposition} The bijection \eqref{isom} is a biholomorphism of complex manifolds.\end{proposition}
\begin{proof} Since $\Jac^{g+k-1}(C)$ is smooth, so is $F_0^{g+k-1}(C)$. Let  $\phi=\phi_{L,\underline{s}}$ be an element of $\Mor^P_0(C,\oP^{k-1})$. The pullback of the Euler sequence on $\oP^{k-1}$ is
\begin{equation} 0\rightarrow \sO_C\to H^0(C,L)^\ast\otimes L\to \phi^\ast T_{\oP^{k-1}}\to 0.\label{Euler}\end{equation}
Since $h^1(L)=0$, $h^1(\phi^\ast T_{\oP^{k-1}})=0$, and hence $\Mor_0^p(C,\oP^{k-1})$ is smooth at $\phi$ (see e.g.\ \cite{Deb}). Moreover, the tangent space at $\phi$ is canonically identified  with $H^0(C,\phi^\ast T_{\oP^{k-1}})$, and the long exact sequence of of \eqref{Euler} yields
\begin{equation}0\to \cx \longrightarrow H^0(C,L)^\ast\otimes H^0(C,L)\longrightarrow  H^0(C,\phi^\ast T_{\oP^{k-1}})\longrightarrow H^1(C,\sO_C)\to 0.\label{TMor}\end{equation}
The subspace $ H^0(C,L)^\ast\otimes H^0(C,L)/\cx$ of $ H^0(C,\phi^\ast T_{\oP^{k-1}})$ is precisely the subspace generated by fundamental vector fields (corresponding to a change of frame), while  $H^1(C,\sO_C)$ corresponds to variations of the invertible sheaf $L$. This describes the differential of  \eqref{isom}  and shows that it is an isomorphism at every point of $F_0^{g+k-1}(C)$.
\end{proof}

\smallskip

We now assume that the arithmetic genus of $C$ is equal to $(k-1)^2$, and that $C$ is equipped with a fixed flat projection $\pi:C\to\oP^1$ of degree $k$.  If $C$ is smooth, then the number of branch points is $2k^2-2k$, and so the dimension of the moduli space parametrising such pairs $(C,\pi)$ is $2k^2-2k$ (more precisely this is the dimension of the Hurwitz scheme parametrising simple branch coverings of $\oP^1$ of degree $k$ with $2k^2-2k$ branch points).
\par
Denote by $\sO_C(1)$ the line bundle $\pi^\ast \sO_{\oP^1}(1)$. This is an invertible sheaf of degree $k$, and the map $L\mapsto L(1)$ is an isomorphism $\Jac^{g-1}(C)\to \Jac^{g+k-1}(C)$. We denote by $ \Jac_\ast^{g+k-1}(C)$ the subset consisting of $L$ such that $L(-1)\not\in \Theta_C$. We have the corresponding open subsets $F_\ast^{g+k-1}(C)$ of $F_0^{g+k-1}(C)$ and $\Mor_\ast^p(C,\oP^{k-1})$ of $\Mor_0^p(C,\oP^{k-1})$.
\par
We  now define a natural $PGL(k,\cx)$-invariant integrable $2$-Kronecker structure on  $M=F_\ast^{g+k-1}(C)\simeq \Mor_\ast^p(C,\oP^{k-1})$. The bundle $E$ is 
$$ E\rvert_\phi=H^0\bigl(C,\phi^\ast T_{\oP^{k-1}}\otimes\sO_C(-1)\bigr),$$
and the morphism $\alpha:E\otimes \cx^2\to TM$ is given by multiplication by elements of $\pi^\ast H^0(\oP^1,\sO_{\oP^1}(1))\simeq \cx ^2$. If $s\in H^0(\oP^1,\sO_{\oP^1}(1))$ vanishes at $\zeta$, then we have an exact sequence
$$ 0\to \phi^\ast T_{\oP^{k-1}}\otimes\sO_C(-1)\stackrel{\cdot s\circ \pi}{\longrightarrow}\, \phi^\ast T_{\oP^{k-1}} \longrightarrow \phi^\ast T_{\oP^{k-1}}\rvert_{\pi^{-1}(\zeta)}\to 0,$$
and $\alpha\rvert_{E\otimes\zeta}$ is the induced map on global sections. Hence $\alpha$ defines a $2$-Kronecker structure (of rank $\frac{1}{2}\dim M$). Moreover, the condition $h^1(L(-1))=0$ and \eqref{Euler} imply that $h^1\bigl(C,\phi^\ast T_{\oP^{k-1}}\otimes\sO_C(-1)\bigr)=0$, and hence the 
the image of $E\otimes\zeta$ are precisely those sections of $\phi^\ast T_{\oP^{k-1}}$ which vanish on $\pi^{-1}(\zeta)$. This is the tangent space to the subvariety of $\Mor_\ast^p(C,\oP^{k-1})$ consisting of morphisms which restrict to a fixed morphism over $\pi^{-1}(\zeta)$ (see, e.g.\ \cite{Deb}). The $2$-Kronecker structure is therefore integrable.
\begin{proposition} The above $2$-Kronecker structure is $\cx$-hypercomplex on the subset of $F_\ast^{g+k-1}(C)$ corresponding to $L\in \Jac_\ast^{g+k-1}(C)$ such that the natural multiplication map
$$ H^0(C,L)\otimes H^0(C,L^\ast\otimes K_C(2))\longrightarrow H^0(C,K_C(2))$$
is an isomorphism.\label{F**}
\end{proposition}
\begin{proof} The Euler sequence on $\oP^1$ gives
$$ 0\to \phi^\ast T_{\oP^{k-1}}\otimes\sO_C(-2) \longrightarrow \phi^\ast T_{\oP^{k-1}}\otimes\sO_C(-1)\otimes \cx^2 \longrightarrow \phi^\ast T_{\oP^{k-1}}\otimes\sO_C\to 0,$$
and hence $\alpha$ is an isomorphism if and only if $ h^1\bigl(\phi^\ast T_{\oP^{k-1}}\otimes \sO_C(-2)\bigr)=0$ (since $\chi\bigl(\phi^\ast T_{\oP^{k-1}}\otimes \sO_C(-2)\bigr)=0$ from \eqref{Euler}).
After tensoring \eqref{Euler} by $\sO_C(-2)$ and taking the long exact sequence, this is equivalent to the natural map
$$ m: H^1(C,\sO_C(-2))\to H^0(C,L)^\ast\otimes H^1(C,L(-2))
\simeq \Hom\bigl( H^0(C,L), H^1(C,L(-2)\bigr)
$$
being an isomorphism. 
Observe that this map is simply $ m(\phi)(s)=\phi s$. Using Serre duality, the dual map is the multiplication
\begin{equation*} H^0(C,L)\otimes H^0(C,K_CL^\ast(2))\to H^0(C,K_C(2)).\label{multipl}\end{equation*}
\end{proof}
\begin{example} The $2$-Kronecker structure is certainly degenerate if $L(-1)$ is a theta characteristic. Indeed, in this case $K_CL^\ast(2)\simeq L$ and the multiplication map $H^0(C,L)\otimes H^0(C,L)\to H^0(C,K_C(2))$ vanishes on the skew-symmetric part of the tensor product.
\end{example}

\subsection{Reality conditions\label{M(S,pi)}}

Suppose now that $C$ is equipped with an antiholomorphic involution $\sigma$ covering the antipodal map on $\oP^1$. We consider the following induced involution on $\Jac^{g+k-1}(C)$:
$$\sigma: L\mapsto \bigl(\ol{\sigma^\ast L}\bigr)^\ast\otimes K_C(2),$$
and $L$ will be called {\em real} if $\sigma(L)\simeq L$, i.e.\ $\ol{\sigma^\ast L}\simeq L^\ast \otimes K_C(2)$.
\par
If $L$ is a real, then the pullback $s\to\sigma^\ast s$ induces an antiholomorphic isomorphism
\begin{equation} H^0(C,L)\to H^0(C,K_CL^\ast (2)).\label{sigma-sect}\end{equation}
In order to obtain a hypercomplex structure, we need to extend this real structure to $F_\ast^{g+k-1}(C)$. If we tensor the pullback to $C$ of the Euler sequence on $\oP^1$ with $K_C$, we obtain
\begin{equation} 0\rightarrow K_C\to K_C(1)\otimes \cx^2\to K_C(2)\to 0,\label{K_S}
\end{equation}
and hence a natural surjective homomorphism
\begin{equation} \gamma:H^0(C,K_C(2))\to H^1(C,K_C)\simeq \cx.
\end{equation}
The sequence \eqref{K_S} is compatible with the real structure and therefore we can define,
for a real invertible sheaf $L\in \Jac^{g+k-1}(C)$, a hermitian form on $ H^0(C,L)$ by
\begin{equation}\langle s,t\rangle=\gamma(s\sigma^\ast t).\label{hform}\end{equation}
\begin{lemma} If $L(-1)\not\in\Theta_C$, then $\langle\,,\rangle$ is nondegenerate.
\end{lemma}
\begin{proof} Let  $s\in H^0(C,L)$. Let $\zeta_0$ be a point such that $\pi^{-1}(\zeta_0)$ consists of distinct points and  $s$ does not vanish at any of them. We can find a section $t$ of $L$ which does not vanish at exactly one point of $\pi^{-1}(-1/\bar\zeta_0)$. Then 
$u=s\sigma^\ast t$ does not vanish at exactly one point of $\pi^{-1}(\zeta_0)$. If  $\gamma(u)=0$, then \eqref{K_S} implies that $u=v_1+(\zeta-\zeta_0)v_2$ for some $v_1,v_2\in H^0(C,K_C(1))$. Therefore $v_1$ does not vanish at exactly one point $p$ of $\pi^{-1}(\zeta_0)$. But this means that  $v_1\in H^0(C,K_C(p))\backslash H^0(C,K_C)=\emptyset$.
This contradiction implies that $\gamma(s\sigma^\ast t)\neq 0$.
\end{proof}
We thus obtain a quaternionic $2$-Kronecker structure on the subset $V$ of $F_\ast^{g+k-1}(C)$ where the sheaves are real and the fibres consist of unitary frames. Since we are interested in $PU(k)$-invariant hypercomplex structure,
we denote by $\bar M(C,\pi)$ the subset of $V$ where \eqref{hform} is either positive- or negative-definite. Its open subset, where the condition of Proposition \ref{F**} is satisfied will be denoted by $M(C,\pi)$.
\par
Thus  $M(C,\pi)$ (resp.\  $\bar M(C,\pi)$) is an $U(k)$-invariant hypercomplex manifold (resp. $U(k)$-invariant quaternionic $2$-Kronecker manifold).
\begin{remark} $M(C,\pi)$ must be nonempty at least for $(C,\pi)$ in a neighbourhood of $\sM_0$ inside the real locus of the Hurwitz scheme $\sH_{k,(k-1)^2}$, where $\sM_0$ consists of real $(C,\pi)$ such that $C$ is embedded in $T\oP^1$ and $\pi$ is the restriction of $T\oP^1\to \oP^1$.
\end{remark}


\subsection{Relation to D'Amorim Santa-Cruz manifolds}
D'Amorim Santa-Cruz's manifolds $M(S)$ (cf.\ \S\ref{SC}) for $G=U(k)$ are included in the above construction and correspond to pairs $(C,\pi)$ such that $C=S$ is embedded in $T\oP^1$ (with the original projection $\pi:C\to\oP^1$ equal to the one arising from the embedding). Indeed, if $C$ is embedded in $T\oP^1$, then it follows from results of Beauville \cite{Beau} that $\Jac_\ast^{g+k-1}(C)$ corresponds to $GL(k,\cx)$-conjugacy classes of quadratic $\gl(k,\cx)$-valued polynomials. The $\cx$-hypercomplex structures on the space of strongly regular sections and on the manifold defined in Proposition \ref{F**} are easily seen to be the same, as are their restrictions to the real parts. The hermitian form $\langle\,,\rangle$ on $ H^0(C,L)$ is then that of Hitchin \cite[\S 6]{Hit}.
\par
Proposition \ref{F**} provides now an algebro-geometric characterisation of strong regularity. A quadratic polynomial defines a strongly regular section if and only if  the corresponding invertible sheaf satisfies the condition of Proposition \ref{F**}.

\subsection{Complex structures}

We return to the hypercomplex manifold $M(C,\pi)$ defined above. It consists of $\sigma$-invariant nondegenerate morphisms $\pi:C\to \oP^{k-1}$ which correspond to invertible sheaves $L$ of degree $g+k-1$ such that
\begin{itemize}
\item[(i)] $L(-1)\not\in \Theta_C$;
\item[(ii)] the multiplication map $H^0(C,L)\otimes H^0(C,K_CL^\ast(2))\to H^0(C,K_C(2))$ is an isomorphism;
\item[(iii)] the hermitian form $\langle\,,\rangle$ is definite.
\end{itemize}

For every $\zeta\in\oP^1$ the fibre $C_\zeta=\pi^{-1}(\zeta)$ of $\pi:C\to\oP^1$ can be identified with a {\em fat point} $\sum k_ip_i$ in $\cx$, i.e.\ a zero-dimensional scheme $\bigcup_{i=1}^s {\rm Spec}\, \cx[t]/((t-t_i)^{k_i})$, with $t_i$ distinct. The restriction of $\phi$ to $C_\zeta$ must be nondegenerate, since the image of $\phi_\zeta$ being contained in a hyperplane is equivalent to a section of $L$ vanishing on $\pi^{-1}(\zeta)$, which contradicts (i) above.  Thus, for any complex structure $I_\zeta\in Q$ of $M(C,\pi)$ corresponding to $\zeta\in \oP^1$, we have a local equivariant biholomorphism from $M(C,\pi)$ to the manifold of nondegenerate morphisms from $\sum k_ip_i$ to $\oP^{k-1}$. This latter manifold is easily described:
\begin{lemma} The (smooth) space of nondegenerate morphisms from $\sum k_ip_i$ to $\oP^{k-1}$ is equivariantly isomorphic to $GL(k,\cx)/Z(J)$, where $Z(J)$ is the centraliser of the  Jordan normal form matrix $J=J_{t_1,k_1}\oplus\dots \oplus J_{t_s,k_s}$, with distinct $t_i\in \cx$.\label{Z}\end{lemma}
\begin{proof}  The image of each multiple point $k_ip_i$ can span at most a $(k_i-1)$-hyperplane in $\oP^{k-1}$. Therefore, if $\phi:\sum k_ip_i\to\oP^{k-1}$ is nondegenerate, then there exists a direct sum decomposition $\cx^k=\bigoplus_{i=1}^s V_i$, $\dim V_i=k_i$, such that $\phi|_{k_ip_i}$ is a nondegenerate morphisms to $\oP(V_i)$. Using the action of $GL(k,\cx)$ we can assume that each $V_i$ is spanned by the corresponding set of coordinate vectors. Using now the action of $GL(V_i,\cx)$ we can move $\phi|_{ k_ip_i}$ to $t\mapsto [1,t,\dots, t^{k_i-1}]$ in $\oP V_i$. Thus the space of nondegenerate morphisms from $\sum k_ip_i$ to $\oP^{k-1}$ is acted upon transitively by $GL(k,\cx)$ and the stabiliser of 
$$t\longmapsto [1,t,\dots, t^{k_1-1}, 1,t,\dots, t^{k_2-1},\dots\dots, 1,t,\dots, t^{k_s-1}]$$ consists of linear transformations $(g_1,\dots,g_s)\in \bigoplus_{i=1}^sGL(k_i,\cx)$ such that $$g_i(1,t,\dots,t^{k_i-1})^T= \alpha_i(t)(1,t,\dots,t^{k_i-1})^T,\quad i=1,\dots,s,$$ for a polynomial $\alpha_i(t)$ of degree $k_i-1$. The result follows.
\end{proof}  
Therefore $\bigl(M(C,\pi),I_\zeta\bigr)$  is  locally equivariantly  biholomorphic to  $ GL(k,\cx)/Z(J)$. This latter manifold is biholomorphic to  the adjoint orbit of $J$, but the biholomorphism is not canonical. It becomes canonical only in the case of D'Amorim Santa-Cruz manifolds, where there is also a complex symplectic form and the corresponding moment map embedding $\bigl(M(C,\pi),I_\zeta\bigr)$ into $\gl(k,\cx)$.

\section{Proof of Theorem A}

We have already shown how to construct a hypercomplex structure with required properties from an algebraic curve. Conversely, let $M$ be a hypercomplex manifold with properties stated in Theorem A. We view $M$ as a $U(k)$-invariant manifold (with the centre acting trivially) and, for any  $m\in M$ and for every complex structure $I_\zeta\in Q$, $\zeta\in \oP^1$, we  denote by $\fZ_{m,\zeta}$ the infinitesimal stabiliser of $m$ in $\gl(k,\cx)$. Since the antipodal map on $\oP^1$ corresponds to $I\mapsto -I$ on $Q$, we have the following reality condition:
\begin{equation} \fZ_{m,-1/\bar\zeta}=(\fZ_{m,\zeta})^\ast,\quad \forall \zeta\in\oP^1.\label{Zreal}\end{equation}
\begin{lemma} The family $\fZ_{m,\zeta}$, $\zeta\in\oP^1$, is a holomorphic vector bundle  of rank $k$ and degree $k-k^2$.
\end{lemma}
\begin{proof} $\fZ_{m,\zeta}$ is the kernel of the map 
$$ \gl(k,\cx)=\u(k)\oplus i\u(k)\longrightarrow T_mM,\quad \rho_1+i\rho_2\longmapsto X_{\rho_1}+I_\zeta X_{\rho_2},$$
where $X_\rho$ is the fundamental vector field corresponding to $\rho\in\u(k)$. This map is $\cx$-linear with respect to $I_\zeta$. As $\zeta$ varies, the complex spaces $(TM,I_\zeta)$ form the vector bundle $\sO(1)\otimes \cx^n$ on $\oP^1$, where $n=k^2-k=\dim_\cx M$. Thus the family in the statement is the kernel $\sK$ of $\gl(k,\cx)\to \sO(1)\otimes \cx^n\to 0$. Since the stalks of $\sK$ all have dimension $k$, $\sK$ is locally trivial. The rank and the degree follow from the exact sequence defining $\sK$. 
\end{proof}
The centraliser of a regular element in $\gl(k,\cx)$ is a commutative subalgebra of $\Mat_{k,k}(\cx)$ (with respect to matrix multiplication). Thus the vector bundle $\fZ_m$ defined in the above lemma is a  sheaf of $k$-dimensional commutative algebras, locally free as a sheaf of $\sO_{\oP^1}$-modules. Therefore its ${Spec}$ is an algebraic curve $C_m$ with a flat projection $\pi_m:C_m\to \oP^1$ of degree $k$. The vector bundle $\fZ_m$ is then $\pi_\ast \sO_{C_m}$. Since the algebras $\fZ_{m,\zeta}$ are isomorphic for all $m$, so are the curves $C_m$ and the projections 
$\pi_m:C_m\to \oP^1$. We denote by $\pi:C\to \oP^1$ the abstract curve and its projection, isomorphic to any $\pi_m:C_m\to \oP^1$. The reality conditions imply that $C$ is equipped with an antiholomorphic involution covering the antipodal map.
\begin{lemma} $C$ is  connected, reduced, lci, and of genus $(k-1)^2$.\end{lemma}
\begin{proof}  Connectedness of $C$ is equivalent to the dimension of the trivial summand of $\fZ_m\simeq \pi_\ast \sO_C$ being equal to $1$. This trivial summand  corresponds to $\rho_1+i\rho_2\in\gl(k,\cx)$ such that  $X_{\rho_1}+I_\zeta X_{\rho_2}=0$ for every $\zeta$. This is equivalent to $X_{\rho_1}=X_{\rho_2}=0$. Since we assumed that $PU(k)$ acts freely, the trivial summand of $\fZ_m$ is the centre of $\gl(k,\cx)$, hence it is $1$-dimensional, and $C$ is connected. Choose now a local section $c$ of $\fZ_m$, which generates
$\fZ_{m,\zeta}$ as algebra. We can then embed $C$ locally into $\cx^2$ as
$$\{(\zeta,\eta)\in \cx^2;\;\det(\eta-c(\zeta))=0\}.$$
This shows that $C$ is lci. It is also reduced, since  $\fZ_{m,\zeta}$  is a Cartan subalgebra for generic $\zeta$, and hence $c(\zeta)$ has genericall $k$ distinct eigenvalues.
Finally, the genus of $C$ is $(k-1)^2$, since $1-g= \chi(\sO_C)=\chi(\fZ_m)$.
\end{proof}
Consider now the vector bundle $E\simeq \sO_{\oP^1}(-1)^{\oplus k}$. Since $\fZ_{m,\zeta}\subset \gl(k,\cx)$, $E$ has the structure of a $\fZ_m$-module, i.e.\ a $\pi_\ast\sO_C$-module. Since $\fZ_{m,\zeta}$ are centralisers of regular elements, $E$ is locally isomorphic to $\pi_\ast\sO_C$, and hence  $E\simeq \pi_\ast L_m$ for an invertible sheaf $L_m$ on $C$. Moreover, $h^0(L_m)=h^1(L_m)=0$, and so $L_m\in \Jac^{g-1}(C)\backslash \Theta$.
\par
We need to show that the sheaves $L_m$ are $\sigma$-invariant in the sense of Theorem A. Suppose that $\fZ$ is a locally free sheaf of centralisers of regular elements of $\gl(k,\cx)$, such that its $Spec$ is isomorphic to $C$. Let $L$ be the corresponding invertible sheaf arising from the $\fZ$-module structure on $E\simeq \sO_{\oP^1}(-1)^{\oplus k}$. Then:
\begin{lemma} The sheaf corresponding to $\fZ^T$ is $L^\ast\otimes K_C$.
\end{lemma}
\begin{proof}
The $\fZ^T$-module structure on $E$ is isomorphic to ${\mathcal Hom}_{\oP^1}(\pi_\ast L,\sO(-2))$ (where $\pi_\ast L$ is, by definition, $E$ with its $\fZ$-module structure). Since
$$ \pi_\ast {\mathcal Hom}_C(L,K_C)\simeq {\mathcal Hom}_{\oP^1}(\pi_\ast L,K_{\oP^1}),$$
the claim follows.
\end{proof}
This lemma and \eqref{Zreal} show that $L_m(1)$ is indeed real for $m\in M$ (in the sense of \S \ref{M(S,pi)}).
Moreover, the proof also implies that the antiholomorphic isomorphism \eqref{sigma-sect} is simply the conjugation on $H^0(\oP^1,E(1))\simeq \cx^k$, and the hermitian form \eqref{hform} is the standard hermitian form on $\cx^k$.  We thus obtain a natural smooth map $\Psi:M\to \bar M(C,\pi)$, where the latter manifold was defined at the end of \S \ref{M(S,pi)}. We claim that $\Psi$ is injective:
\begin{lemma}$L_m\simeq L_{m^\prime}$ if and only if $m$ and $m^\prime$ belong to the same $U(k)$-orbit.
\end{lemma}
\begin{proof} $L_m$ and $L_{m^\prime}$ are isomorphic if and only if the two $\pi_\ast\sO_C$-module structures on $E$ are isomorphic. Given the reality condition \eqref{Zreal} and the fact that $\Aut(E)\simeq GL(k,\cx)$,
we conclude that $L_m\simeq L_{m^\prime}$ if and only if there exists  $g\in U(k)$ such that
$g\,\fZ_{m,\zeta}\,g^{-1}=\fZ_{m^\prime,\zeta}$ for all $\zeta\in \oP^1$. Set $\tilde m=g.m$. Then 
$\fZ_{\tilde m,\zeta} =\fZ_{m^\prime,\zeta}$ for all $\zeta\in \oP^1$. Choose $\zeta$ so that the stabiliser $\fZ_{m^\prime,\zeta}$ is a Cartan subalgebra. Owing to assumption (i) in the theorem, there exists an element $h$ of $GL(k,\cx)$ (for the complex structure $I_\zeta$) such that $h.\tilde m=m^\prime$. This means, however, that $h$ normalises the Cartan subalgebra $\fZ_{m^\prime,\zeta}$, and so $h=wt$, where $t\in \exp\fZ_{m^\prime,\zeta}$ and $w\in S_k\subset U(k)$. Therefore $m$ and $m^\prime$ belong to the same $U(k)$-orbit.
\end{proof}
We can extend $\Psi$ to a neighbourhood of $M$ in $M^\cx$, and it will remain injective there. Since this extension is a holomorphic bijection onto its image, it is a biholomorphism, and hence $\Psi$ is a diffeomorphism onto its image, which must be an open subset of $\bar M(C,\pi)$. The hypercomplex structure of $M$ is also the one constructed in the previous section, and Proposition \ref{F**} implies that the sheaves $L_m$ do not belong to the  divisor $\Delta^\sigma$. The proof is complete.





\end{document}